\documentclass[a4paper, 12pt]{amsart}
\usepackage[cmtip,all]{xy}

\newtheorem{theorem}{Theorem}[section]

\theoremstyle{definition}
\newtheorem{definition}[theorem]{Definition}

\theoremstyle{definition}
\newtheorem{construction}[theorem]{Construction}

\theoremstyle{definition}
\newtheorem{example}[theorem]{Example}

\begin{document}

\date{2008-11-10}

\title[Free joinings]{Free joinings of C*-dynamical systems}

\author{Rocco Duvenhage}

\address{Department of Mathematics and Applied Mathematics\\
University of Pretoria\\
Pretoria 0002\\
South Africa}

\email{rocco.duvenhage@up.ac.za}



\begin{abstract}
Joinings of C*-dynamical systems are defined in terms of free
products of C*-algebras, as an analogue of joinings of classical
dynamical systems. We then consider disjointness in this context, in
particular for ergodic versus identity systems. Lastly we show how
multi-time correlation functions appearing in quantum statistical
mechanics naturally fit into this joining framework.
\end{abstract}

\maketitle

\section{Introduction}

Joinings, and more specifically disjointness, of measure theoretic
dynamical systems were introduced in \cite{F67} and has since become
an important tool in classical ergodic theory (see for example
\cite{dlR} and \cite{G} for overviews). In noncommutative dynamical
systems, in an operator algebraic framework, a generalization of
joinings in terms of tensor products of operator algebras has been
applied in the study of dynamical entropy \cite{ST}, and a more
systematic study of such generalized joinings was initiated in
\cite{D, D2}. An early exploration of disjointness in the
noncommutative setting, but from the point of view of topological
dynamics, can be found in \cite[Section 5]{A2}.

However, although the tensor product approach gives a direct
generalization of classical joinings, it does have its limitations,
for example if one wants to consider noncommutative versions of
so-called graph joinings of more than two copies of the same system.
In \cite[Construction 3.4]{D} and \cite[Section 5]{D2} this problem
could be handled for the case of two copies of the same system by
considering the commutant of the one copy, since this alleviates
commutation problems sufficiently, but for more than two copies this
simple approach does not help. This is unfortunate, as
``self-joinings" of this type have proven very useful in classical
ergodic theory; see for example \cite{Rud, Ryz, dJY} as sample of
papers that have appeared over the years applying this idea.

In this paper we replace the tensor product with a free product of
operator algebras as an alternative way to approach this problem. In
this case we get an analogy rather than a generalization of
classical joinings. We could refer to joinings based on free
products as ``free joinings'', but since we consider only the free
product setting in this paper, no confusion will arise if we simply
use the term ``joinings''. We will however consider two types of
free products, namely the unital C*-algebra free product, and the
reduced free product of unital C*-algebras with specified states,
and will correspondingly use the terms ``joining'' and ``r-joining''
respectively.

Free products of operator algebras were initially studied in
\cite{C, A, V}, and useful sources for the ideas and tools from this
area that we will use are \cite{VDN} and \cite{V2}. We can add that
free products of operator algebras have already appeared in work on
noncommutative ergodic theory. See for example the recent papers
\cite{AD, F, FM}. In particular we will use a result from \cite{AD}
in Section 3.

We discuss the setting and basic definitions and constructions
regarding joinings in Section 2. In Section 3 we consider
disjointness, and in particular how it relates to ergodicity. In the
tensor product case ergodicity can be characterized in terms of
disjointness from identity systems \cite[Theorem 2.1]{D2}, and here
we explore a similar connection in the free product case for certain
classes of systems. In Section 3 we only consider joinings of two
systems at a time, and the goal is to give an idea of how free
joinings compare with tensor joinings in one of the standard
applications of joinings (characterizing ergodicity). Lastly in
Section 4 we briefly study multi-time correlation functions. This
concept has it origins in quantum statistical mechanics \cite{BF,
ABDF}, but here we also motivate it from a mathematical point of
view in terms of higher order mixing of strongly mixing systems
based on reduced group C*-algebras. We show that multi-time
correlation functions and their so-called asymptotic states (if they
exist), are examples of joinings. This illustrates how joinings of
more than two copies of a system occur naturally in applications.

\section{Joinings}

We fix an arbitrary group $G$ and define a \textit{C*-dynamical
system}, or \textit{system} for short, as a
$\mathbf{B}=\left(B,\nu,\beta\right)$ where $B$ is a unital
C*-algebra with a state $\nu$, and $\beta:G\rightarrow$
Aut$(B):g\mapsto\beta^{g}$ is an automorphism group, i.e. a
representation of $G$ as $\ast$-automorphisms of $B$, such that
$\nu\circ\beta^{g}=\nu$ for all $g\in G$. We write $\beta^{g}$
rather than $\beta_{g}$, since we will shortly add other indices and
this will then be a convenient notation; $\beta^{g}$ is simply the
value of the function $\beta$ at $g$. The identity of $G$ will be
denoted by $1$, so for example $\beta^{1}=$ id$_{B}$. We will call
$\mathbf{B}$ an \emph{identity system} if $\beta^{g}=$ id$_{B}$ for
all $g$, and we will call it \emph{trivial} if $B=\mathbb{C}1_{B}$.

Throughout the rest of this section we consider a family
$\left(\mathbf{A}_{\iota}\right)_{\iota\in I}$ of systems (but keep
in mind they all use the same group $G$), where we use the notation
$\mathbf{A}_{\iota}=\left(A_{\iota},\mu_{\iota},\alpha_{\iota}\right)$,
and let $A:=\ast_{\iota\in I}A_{\iota}$ be the unital C*-algebra
free product \cite[Definition 1.4.1]{VDN}, and let
$\mu:=\ast_{\iota\in I}\mu_{\iota}$ be the free product state on $A$
\cite[Definition 1.5.4]{VDN}. Using the universal property of $A$ we
can define a free product $\alpha$ of the automorphism groups as
follows in terms of a commutative diagram:
\[
\xymatrix{A_\iota\ar[r]^{\psi_\iota}\ar[d]_{\alpha_{\iota}^{g}}%
&A\ar[d]^{\alpha^{g}}\\A_\iota\ar[r]_{\psi_\iota}&A}
\]
Here $\psi_{\iota}:A_{\iota}\rightarrow A$ is the injective unital
$\ast $-homomorphism that appears in the universal property of $A$.
Then one can verify that $\mu\circ\alpha^{g}=\mu$ for all $g\in G$,
i.e. $\left(A,\mu,\alpha\right)$ is a system. In terms of this we
give

\begin{definition}
A \emph{joining} of $\left(\mathbf{A}_{\iota}\right)_{\iota\in I}$
is any state $\omega$ on $A$ such that
$\omega\circ\psi_{\iota}=\mu_{\iota}$ for all $\iota\in I$, and such
that $\omega\circ\alpha^{g}=\omega$ for all $g\in G$. Let
$J\left(\left(\mathbf{A}_{\iota}\right)_{\iota\in I}\right)$ be the
set of all such joinings.
\end{definition}

Note that $\mu\in J\left(  \left(  \mathbf{A}_{\iota}\right)
_{\iota\in I}\right)  $ and we call $\mu$ the \emph{trivial} joining
of $\left( \mathbf{A}_{\iota}\right)  _{\iota\in I}$.

We will also consider a second type of joining in terms of the
reduced free
product $\left(  R,\varphi\right)  :=\ast_{\iota\in I}\left(  A_{\iota}%
,\mu_{\iota}\right)  $, the definition of which is discussed for
example in
\cite{V2}. Here one considers the GNS representation $\left(  H_{\iota}%
,\pi_{\iota},\Omega_{\iota}\right)  $ of $\left(
A_{\iota},\mu_{\iota }\right)  $. Denoting $H_{\iota}\ominus\left(
\mathbb{C}\Omega_{\iota }\right)  $ by $H_{\iota}^{\circ}$ and
setting
\[
H:=\mathbb{C}\Omega\oplus\bigoplus_{n\geq1}\left(  \bigoplus_{\iota_{1}%
\neq\iota_{2}\neq...\neq\iota_{n}}H_{\iota_{1}}^{\circ}\otimes...\otimes
H_{\iota_{n}}^{\circ}\right)
\]
(where we can view $\Omega$ as the number $1$ in $\mathbb{C}$ if we
want to be concrete) one can obtain a representation
$\Lambda_{\iota}$ of $A_{\iota}$ on $H$ (see \cite[Definition
1.5.1]{VDN}), and $R$ is the C*-algebra in $B(H)$ generated by
$\bigcup_{\iota\in I}\Lambda_{\iota}\left(  A_{\iota}\right)  $
while $\varphi(a):=\left\langle \Omega,a\Omega\right\rangle $ for
all $a\in R$. Defining a unitary representation $U_{\iota}$ of
$\alpha_{\iota}$ on $H_{\iota}$ by
\[
U_{\iota}^{g}\pi_{\iota}(a)\Omega_{\iota}:=\pi_{\iota}\left(\alpha_{\iota}
^{g}(a)\right)  \Omega_{\iota}
\]
for all $a\in A_{\iota}$ and all $g\in G$, we obtain a unitary
representation $U$ of $G$ on $H$ by setting $U^{g}\Omega:=\Omega$
and $U^{g}\left(x_{1} \otimes...\otimes x_{n}\right)
:=\left(U_{\iota_1}^{g}x_{1}\right)\otimes...\otimes\left(U_{\iota_n}^{g}x_{n}\right)$
for all elementary tensors $x_{1}\otimes...\otimes x_{n}\in
H_{\iota_{1}}^{\circ}\otimes ...\otimes H_{\iota_{n}}^{\circ}$ and
all $g\in G$. Setting $\pi:=\ast_{\iota\in
I}\Lambda_{\iota}:A\rightarrow B(H)$ it can be shown that
\[
U^{g}\pi(a)\Omega=\pi(\alpha^{g}(a))\Omega
\]
for all $a\in A$ and $g\in G$. It can furthermore be shown that
\[
\rho^{g}(a):=U^{g}a\left(U^{g}\right)^{\ast}
\]
gives a well defined automorphism group $\rho:G\rightarrow$ Aut$(R)$
which satisfies $\varphi\circ\rho^{g}=\varphi$ for all $g\in G$. In
other words we have again obtained a system
$\left(R,\varphi,\rho\right)$. In terms of this we give

\begin{definition}
A \emph{reduced free product joining} (or \emph{r-joining} for
short) of $\left(\mathbf{A}_{\iota}\right)_{\iota\in I}$ is any
state $\omega$ on $R$ such that
$\omega\circ\Lambda_{\iota}=\mu_{\iota}$ for all $\iota\in I$, and
such that $\omega\circ\rho^{g}=\omega$ for all $g\in G$. Let
$J_{r}\left( \left(  \mathbf{A}_{\iota}\right)  _{\iota\in I}\right)
$ be the set of all such joinings.
\end{definition}

Note that $\varphi\in J_{r}\left(  \left(  \mathbf{A}_{\iota}\right)
_{\iota\in I}\right)  $ and we call $\varphi$ the \emph{trivial}
r-joining of $\left(  \mathbf{A}_{\iota}\right)  _{\iota\in I}$.

In the case of joinings we now show how to construct an analogue of
(a class of) graph joinings that appear in classical ergodic theory.
We will use this construction in the subsequent sections.

\begin{construction}
Assume that the systems $\left( \mathbf{A}_{\iota}\right) _{\iota\in
I}$ are \emph{factors} of some system $\mathbf{B=}\left(
B,\nu,\beta\right)  $, i.e. there are unital $\ast$-homomorphisms
$h_{\iota }:A_{\iota}\rightarrow B$ such that $\nu\circ
h_{\iota}=\mu_{\iota}$ and $\beta^{g}\circ
h_{\iota}=h_{\iota}\circ\alpha_{\iota}^{g}$. (A simple instance of
this is $\mathbf{A}_{\iota}=\mathbf{B}$ for all $\iota$, in which
case we will end up with a so-called self-joining, since it will be
a joining of copies of $\mathbf{B}$.) Now we use the universal
property of $\mathbf{A}$ to define a unital $\ast$-homomorphism
$\delta:A\rightarrow B$ by $\delta\circ\psi_{\iota}=h_{\iota}$ for
all $\iota\in I$, i.e. by the following commutative diagram:
\[
\xymatrix{A_\iota\ar[r]^{\psi_\iota}\ar[d]_{h_\iota}&A\ar[dl]^\delta\\B}
\]
Then we set
\[
\Delta:=\nu\circ\delta
\]
which is a joining of $\left(  \mathbf{A}_{\iota}\right)  _{\iota\in
I}$ as we now show.

Clearly $\Delta$ is a state on $A$, and
$\Delta\circ\psi_{\iota}=\nu\circ h_{\iota}=\mu_{\iota}$. One can
view $\Delta$ as an analogue of a diagonal measure. Combining the
diagram for $\delta$ with the diagram for $\alpha$, we obtain the
commutative diagram
\[
\xymatrix{A_\iota\ar[r]^{\psi_\iota}\ar[d]_{h_\iota\circ\alpha^g_\iota}&A\ar[dl]^{\delta\circ\alpha^g}\\B}
\]
However, combining $\delta$ 's diagram with
\[
\xymatrix{B&&A_\iota\ar[d]^{h_\iota}\ar[ll]_{h_\iota\circ\alpha^g_\iota}\\&&B\ar[ull]^{\beta^g}}
\]
we obtain the commutative diagram
\[
\xymatrix{A_\iota\ar[r]^{\psi_\iota}\ar[d]_{h_\iota\circ\alpha^g_\iota}&A\ar[dl]^{\beta^g\circ\delta}\\B}
\]
It follows from the two diagrams that we have just obtained,
together with the universal property, that
$\beta^{g}\circ\delta=\delta\circ\alpha^{g}$, and therefore
$\Delta\circ\alpha^{g}=\nu\circ\beta^{g}\circ\delta=\nu\circ
\delta=\Delta$. Hence $\Delta$ is indeed a joining of $\left(  \mathbf{A}%
_{\iota}\right)  _{\iota\in I}$.

Assuming that $G$ is abelian, we can take this construction further
by considering a family $\bar{g}:=\left(  g_{\iota}\right)
_{\iota\in I}$ of
elements of $G$, and defining a unital $\ast$-homomorphism $\delta_{\bar{g}%
}:A\rightarrow B$ by
$\delta_{\bar{g}}\circ\psi_{\iota}=h_{\iota}\circ
\alpha_{\iota}^{g_{\iota}}$ using the universal property of $A$. We
then set
\[
\Delta_{\bar{g}}:=\nu\circ\delta_{\bar{g}}%
\]
which is again a joining by similar arguments as for $\Delta$: We
obtain $\left(  \delta_{\bar{g}}\circ\alpha^{g}\right)
\circ\psi_{\iota}=h_{\iota }\circ\alpha_{\iota}^{g_{\iota}g}$ from
$\alpha$ and $\delta_{\bar{g}}$ 's diagrams, and $\left(
\beta^{g}\circ\delta_{\bar{g}}\right)  \circ\psi
_{\iota}=h_{\iota}\circ\alpha^{gg_{\iota}}$ from $\delta_{\bar{g}}$
's diagram
combined with $\beta^{g}\circ\left(  h_{\iota}\circ\alpha_{\iota}^{g_{\iota}%
}\right)  =h_{\iota}\circ\alpha_{\iota}^{gg_{\iota}}$. Since $G$ is
abelian, it follows that
$\beta^{g}\circ\delta_{\bar{g}}=\delta_{\bar{g}}\circ \alpha^{g}$.
\end{construction}

As we will see in the next section (in Theorem 3.3's proof), the
joinings in Construction 2.3 are generally not trivial. Simple
nontrivial r-joinings will also be discussed in the next section.

Another construction that will be used in the next section is the
following:

\begin{construction}
Let $\omega$ be any state on $A$ such that $\omega\circ
\psi_{\iota}=\mu_{\iota}$ and let
$\left(H_{\omega},\pi_{\omega},\Omega_{\omega}\right)$ be the GNS
representation of $\left(A,\omega\right)$. Set
$\gamma_{\omega}:=\pi_{\omega}\left( \cdot\right) \Omega_{\omega}$
and $\gamma_{\iota}:=\gamma_{\omega}\circ\psi_{\iota}$, and let
$H_{\iota}$ be the closure of $\gamma_{\iota}\left(A_{\iota}\right)$
in $H_{\omega}$, so we in effect obtain
$\gamma_{\iota}:A_{\iota}\rightarrow H_{\iota}$. It is easily seen
that we can take the GNS representations of
$\left(A_{\iota},\mu_{\iota}\right)$ 's discussed earlier to be
given by $\Omega_\iota:=\Omega_\omega$ and
$\pi_{\iota}(a)\gamma_\iota(b)=\gamma_\iota(ab)$ on the $H_{\iota}$
that we have just obtained.

Denote the projection of $H_{\omega}$ onto $H_{\iota}$ by
$P_{\iota}$ and set $P_{\iota}^{\kappa}:=P_{\iota}|_{H_{\kappa}}$
for all $\iota,\kappa\in I$. It is easy to verify that
$P_{\iota}^{\kappa}$ is the unique function $H_{\kappa }\rightarrow
H_{\iota}$ satisfying $\left\langle P_{\iota}^{\kappa
}x,y\right\rangle =\left\langle x,y\right\rangle $ for all $x\in
H_{\kappa}$ and $y\in H_{\iota}$, and we can call it a
\emph{conditional expectation operator}.

Assume furthermore that $\omega$ is a joining of
$\left(\mathbf{A}_{\iota }\right)_{\iota\in I}$, define a unitary
representation $U_{\omega}$ of $\alpha$ on $H_{\omega}$ by
\[
U_{\omega}^{g}\pi_{\omega}(a)\Omega_{\omega}:=\pi_{\omega}(\alpha_{g}(a))\Omega_{\omega}
\]
and let $U_{\iota}$ be defined as before. (We can note that if
$\omega=\mu$, then $\left(
H_{\omega},\pi_{\omega},\Omega_{\omega}\right)  $ will be (unitarily
equivalent to) $\left(  H,\pi,\Omega\right)  $ which we defined
earlier.) It is straightforward to show that
$U_{\omega}^{g}|_{H_{\iota}}=U_{\iota}^{g}$ for all $g$ and $\iota$.

Combining all this we find that
\[
P_{\iota}^{\kappa}U_{\kappa}^{g}=U_{\iota}^{g}P_{\iota}^{\kappa}
\]
for all $g$, $\iota$ and $\kappa$, since
$\left\langle\left(U_{\iota}^{g}\right)
^{\ast}P_{\iota}^{\kappa}U_{\kappa}^{g}x,y\right\rangle
=\left\langle U_{\kappa}^{g}x,U_{\iota}^{g}y\right\rangle
=\left\langle U^{g}x,U^{g}y\right\rangle =\left\langle
x,y\right\rangle =\left\langle P_{\iota}^{\kappa}x,y\right\rangle$
for all $x\in H_\kappa$ and $y\in H_\iota$.
\end{construction}

\section{Disjointness and ergodicity}

In classical ergodic theory disjointness of two systems refers to
them having only one joining, for example any ergodic system is
disjoint from all identity systems and this in fact characterizes
ergodicity. The same situation holds in the noncommutative case in
terms of tensor product joinings. In this section we study analogous
results in the free product case. We begin with the relevant
definitions in terms of the notation in Section 2, with
$I=\left\{1,2\right\}$. We will study these two definitions in turn
for two different classes of dynamical systems.

\begin{definition}
Two systems $\mathbf{A}_{1}$ and $\mathbf{A}_{2}$ are called
\emph{tensorially disjoint} if for every $\omega\in
J(\mathbf{A}_{1},\mathbf{A}_{2})$ one has
$\omega(a_{1}a_{2})=\mu_{1}(a_{1})\mu_{2}(a_{2})$, or equivalently
$\omega(a_{2}a_{1})=\mu_{2}(a_{2})\mu_{1}(a_{1})$, for all $a_{1}\in
A_{1}$ and $a_{2}\in A_{2}$.
\end{definition}

\begin{definition}
Two systems $\mathbf{A}_{1}$ and $\mathbf{A}_{2}$ are called
\emph{r-disjoint} if $J_{r}(\mathbf{A}_{1},\mathbf{A}_{2})=\left\{
\varphi\right\}  $.
\end{definition}

We call the system $\mathbf{A_\iota}$ \emph{ergodic} if
\[
\left\{x\in H_\iota:U_\iota^{g}x=x\text{ for all }g\in
G\right\}=\mathbb{C}\Omega_\iota
\]

In Theorem 3.3 below we have to assume additional structure, namely
that the systems involved are actually W*-dynamical systems. The
system $\mathbf{A_\iota}$ is called a \emph{W*-dynamical system} if
$A_\iota$ is a $\sigma$-finite von Neumann algebra and $\mu_\iota$
is a faithful normal state. For such a system ergodicity is
equivalent to the \emph{fixed point algebra}
\[
A_\iota^{\alpha_\iota }:=\left\{a\in
A_\iota:\alpha_\iota^{g}(a)=a\text{ for all }g\in G\right\}
\]
being trivial, i.e. $A_\iota^{\alpha_\iota}=\mathbb{C}1_{A_\iota}$,
according to \cite[Theorem 4.3.20]{BR}.

\begin{theorem}
A W*-dynamical system is ergodic if and only if it is tensorially
disjoint from all identity W*-dynamical systems.
\end{theorem}

\begin{proof}
Let $\mathbf{A}_{1}$ be an ergodic W*-dynamical system, and
$\mathbf{A}_{2}$ an identity W*-dynamical system, i.e.
$\alpha_{2}^{g}=$ id$_{A_{2}}$ for all $g\in G$. Consider any
$\omega\in J\left(\mathbf{A}_{1},\mathbf{A}_{2}\right)$ and apply
Construction 2.4 to see that for any $a_{2}\in A_{2}$,
\[
U_{1}^{g}P_{1}^{2}\gamma_{2}(a_{2})=P_{1}^{2}U_{2}^{g}\gamma_{2}(a_{2})=P_{1}^{2}\gamma_{2}(a_{2})
\]
since $\mathbf{A}_{2}$ is an identity system. However, since
$\mathbf{A}_{1}$ is ergodic, the fixed point space of the unitary
group $U_{1}$ is $\mathbb{C}\Omega_{\omega}$, and therefore
$P_{1}^{2}\gamma_{2}(a_{2})=\mu_{2}(a_{2})\Omega_{\omega}$. Hence
\begin{align*}
\omega(a_{1}a_{2})
& =\left\langle\gamma_{1}(a_{1}^{\ast}),\gamma_{2}(a_{2})\right\rangle \\
& =\left\langle\gamma_{1}(a_{1}^{\ast}),P_{1}^{2}\gamma_{2}(a_{2})\right\rangle \\
& =\mu_{1}(a_{1})\mu_{2}(a_{2})
\end{align*}
for all $a_{1}\in A_{1}$ and $a_{2}\in A_{2}$ as required.

Conversely, suppose the W*-dynamical system $\mathbf{A}_{1}$ is not
ergodic, and set $A_2:=A_1^{\alpha_1}$ and
$\mathbf{A}_{2}:=\left(A_2,\mu_{1}|_{A_2},\text{id}_{A_2}\right)$.
Then $\mathbf{A}_{2}$ is a nontrivial identity W*-dynamical system
and also a factor of $\mathbf{A}_{1}$ via the embedding
$h_{2}:A_{2}\rightarrow A_{1}$. Apply Construction 2.3 to obtain the
joining $\Delta\in J\left(\mathbf{A}_{1},\mathbf{A}_{2}\right) $,
where we have set $\mathbf{B}:=\mathbf{A}_{1}$ and $h_{1}:=$
id$_{A_{1}}$. Now note that we do not have
$\Delta(a_{1}a_{2})=\mu_{1}(a_{1})\mu_{2}(a_{2})$ for all $a_{1}\in
A_{1}$ and $a_{2}\in A_{2}$, for if we did it would follow that
\begin{align*}
\left\langle \pi_{1}(a_{1})\Omega_{1},\pi_{1}(h_{2}(a_{2}))\Omega_{1}\right\rangle
& =\Delta(a_{1}^{\ast}a_{2})\\
& =\mu_{1}(a_{1}^\ast)\mu_{2}(a_{2})\\
& =\left\langle \pi_{1}(a_{1})\Omega_{1},\mu_{2}(a_{2})\Omega_{1}%
\right\rangle
\end{align*}
hence $\pi_{1}(h_{2}(a_{2}))\Omega_{1}=\mu_{2}(a_{2})\Omega_{1}$,
but $\Omega_{1}$ is separating for $\pi_{1}(A_{1})$ (since $\mu_{1}$
is faithful) therefore
$\pi_{1}(h_{2}(a_{2}))=\mu_{2}(a_{2})1_{B(H_1)}$. Since $\pi_{1}$
and $h_{2}$ are injective, we conclude that
$a_{2}=\mu_{2}(a_{2})1_{A_{2}}$ which contradicts the fact that
$\mathbf{A}_{2}$ is not trivial. This proves that $\mathbf{A}_{1}$
and $\mathbf{A}_{2}$ are not tensorially disjoint.
\end{proof}

Note that the second part proof of Theorem 3.3 provides an instance
of a joining $\Delta$ which is not trivial (when $\mathbf{A}_{1}$ is
not ergodic).

The proof of Theorem 3.3 is very similar to the tensor product case
in \cite{D, D2}, but somewhat simpler, since in in the tensor
product analogue of Construction 2.3, namely \cite[Construction
3.4]{D}, we had to resort to Tomita-Takesaki theory. The result
itself is of course not quite the same as the tensor product case,
since we have not been able to prove that an ergodic W*-dynamical
system only has the trivial joining with any identity W*-dynamical
system.

In the remainder of this section we look at this last problem form
the perspective of r-joinings, but only for a very special class of
systems. We find that for this class of systems ergodicity implies
r-disjointness from identity systems (i.e. only the trivial
r-joining occurs), which is a better analogy with the tensor product
case (including the classical case).

In the rest of this section (and the next) we only consider
$G=\mathbb{Z}$. We now consider systems $\mathbf{A}_{1}$ and
$\mathbf{A}_{2}$ of the following sort: Let $\Gamma_{\iota}$ be the
free group on the alphabet of symbols $S_{\iota}$ and let
$T_{\iota}$ be an automorphism of $\Gamma_{\iota}$ obtained from a
bijection of $S_{\iota}$. We set $H_{\iota}:=l^{2}(\Gamma_{\iota})$.
Let $A_{\iota}$ be the reduced group C*-algebra
$C_{r}^{\ast}\left(\Gamma_{\iota}\right)$ and define
$\mu_{\iota}(a):=\left\langle\Omega_{\iota},a\Omega_{\iota}\right\rangle$
where $\Omega_{\iota}\in H_{\iota}$ is defined by
$\Omega_{\iota}(1)=1$ and $\Omega_{\iota}(g)=0$ for $g\neq1$. Using
the unitary operator $U_{\iota }:H_{\iota}\rightarrow
H_{\iota}:f\mapsto f\circ T_{\iota}^{-1}$ we obtain a well-defined
$\ast$-automorphism $\alpha_{\iota}:A_{\iota}\rightarrow
A_{\iota}:a\mapsto U_{\iota}aU_\iota^{\ast}$ which of course gives
an automorphism group $\mathbb{Z}\ni n\mapsto\alpha_{\iota}^{n}$
which we also simply denote as $\alpha_{\iota}$. This gives a system
$\mathbf{A}_{\iota}=\left(  A_{\iota
},\mu_{\iota},\alpha_{\iota}\right)  $ which we will call a
\emph{group system}. Note that for the generators
$\lambda_{\iota}(g)$ of $A_{\iota}$ given by the left regular
representation $\lambda_{\iota}$ of $\Gamma_{\iota}$ (defined as
$\left[\lambda_{\iota}(g)f\right](h):=f(g^{-1}h)$ for $f\in
H_{\iota}$ and $g,h\in\Gamma_{\iota}$) one has the simple relation
$\alpha_{\iota}(\lambda_{\iota}(g))=\lambda_{\iota}\left(T_{\iota}g\right)$.
We will consider these systems in the next section as well. The
group system $\mathbf{A}_\iota$ is ergodic (as defined earlier) if
and only if the orbits $\left(T_\iota^{n}g\right)_{n\in\mathbb{Z}}$
are infinite for all $g\in\Gamma_\iota\backslash\{1\}$.

\begin{theorem}
Let $\mathbf{A}_{1}$ and $\mathbf{A}_{2}$ be group systems as above,
with
$\mathbf{A}_{1}$ ergodic and $S_{1}$ countably infinite, while $\mathbf{A}%
_{2}$ is an identity system and $S_{2}$ is finite or countably
infinite. Then $\mathbf{A}_{1}$ and $\mathbf{A}_{2}$ are r-disjoint.
\end{theorem}

\begin{proof}
The key point of this proof is a recent result by Abadie and Dykema
\cite[Proposition 3.5]{AD} regarding unique ergodicity relative to
fixed point algebras. In the notation of Section 2, namely $\left(
R,\varphi\right) :=\left(  A_{1},\mu_{1}\right)  \ast\left(
A_{2},\mu_{2}\right)  $, we have $R=C_{r}^{\ast}\left(
\Gamma_{1}\ast\Gamma_{2}\right)  $ (see for example \cite[Example
1.9]{V2}) and we can view $\Gamma_{2}$ as the subgroup $\left\{
g\in\Gamma_{1}\ast\Gamma_{2}:\left(  T_{1}\ast T_{2}\right)
g=g\right\}  $ of $\Gamma_{1}\ast\Gamma_{2}$, and
$A_{2}=C_{r}^{\ast}(\Gamma_{2})$ as the fixed point algebra of
$\rho$. Also keep in mind that $\Gamma_{1}\ast\Gamma_{2}$ is a free
group on a countably infinite alphabet. Since $\varphi$ is invariant
under $\rho$, and can be viewed as an extension of $\mu_{2}$ to $R$,
it follows from \cite[Proposition 3.5]{AD} that $\varphi$ is the
unique $\rho$ invariant state on $R$ which restricts to $\mu_{2}$.
In particular $\varphi$ is the only r-joining of $\mathbf{A}_{1}$
and $\mathbf{A}_{2}$.
\end{proof}

Note that in this proof we did not use the property
$\omega\circ\Lambda _{1}=\mu_{1}$ of a joining $\omega$ of
$\mathbf{A}_{1}$ and $\mathbf{A}_{2}$, but only
$\omega\circ\Lambda_{2}=\mu_{2}$ and $\omega\circ\rho=\omega$. So it
would seem that unique ergodicity relative to the fixed point
algebra is in this situation a stronger property than r-disjointness
from identity group systems. We can also mention that Abadie and
Dykama's result actually applies more generally than the way that we
have used it in Theorem 3.4, but the setting of Theorem 3.4 is a
very concrete situation which illustrates r-disjointness very
clearly.

It is not clear if the converse of Theorem 3.4 holds, since the
proof technique in Theorem 3.3 relies on Construction 2.3, which
doesn't apply in the case of r-joinings. Theorem 3.4 would therefore
not be very interesting if there were not at least complementary
cases of non-ergodic group systems $\mathbf{A}_{1}$ which are not
r-disjoint from identity group systems. So let us provide as a
simple example the other extreme:

\begin{example}
Let $\mathbf{A}_{1}$ and $\mathbf{A}_{2}$ be identity group systems
(and in fact, in this example the relevant groups $\Gamma_{1}$ and
$\Gamma_{2}$ could even be arbitrary, they need not be free, as long
as they are not trivial, i.e. $\Gamma_{1}\neq\{1\}$ and
$\Gamma_{2}\neq\{1\}$). Remember that as in Section 2 the trivial
r-joining is $\phi$ given by $\phi(a)=\left\langle
\Omega,a\Omega\right\rangle$. Our goal is to exhibit a very simple
non-trivial r-joining of $\mathbf{A}_{1}$ and $\mathbf{A}_{2}$.
Since we are working with identity systems, any state $\omega$ on
$R$ is automatically $\rho$ invariant. We only need to check whether
$\omega$ restricts correctly to $A_{1}$ and $A_{2}$. We consider the
following variation on $\phi$: For $h\in\Gamma_{1}\backslash\{1\}$
and $k\in\Gamma_{2}\backslash\{1\}$, we consider
$z:=\delta_{h}\otimes\delta_{k}\in H_{1}^{\circ}\otimes
H_{2}^{\circ}\subset H$  in terms of the notation in Section 2, but
in this example $\delta_{h}$ is the function such that
$\delta_{h}(h)=1$ while $\delta
_{h}(g)=0$ for $g\in\Gamma_{1}\backslash\{h\}$, and similarly for $\delta_{k}%
$. Then set
\[
\eta:=\frac{\Omega+z}{\left\|  \Omega+z\right\|  }
\]
and define the state $\omega$ by
\[
\omega(a):=\left\langle \eta,a\eta\right\rangle
\]
for all $a\in R$. It can be checked, using the definition of
$\Lambda_{\iota}$ in Section 2, that $\omega$ is indeed an r-joining
by first considering it on the generators of $A_{1}$ and $A_{2}$
given by the left regular representations $\lambda_{1}$ and
$\lambda_{2}$ of $\Gamma_{1}$ and $\Gamma_{2}$. It can similarly be
verified that
$\phi\left(\Lambda_1(\lambda_{1}(h))\Lambda_2(\lambda_{2}(k))\right)=0$
while
$\omega\left(\Lambda_1(\lambda_{1}(h))\Lambda_2(\lambda_{2}(k))\right)=1/2$.
So $\omega\neq\phi$ is indeed non-trivial, and therefore
$\mathbf{A}_{1}$ is not r-disjoint from $\mathbf{A}_{2}$.
\end{example}

There has recently been some activity around topics related to
\cite{AD} and to various ergodicity and mixing conditions more
generally, for example \cite{AM, F, FM, M}. It might also be
interesting to explore where exactly various forms of free product
disjointness from identity systems (or from other classes of
systems) fit into the hierarchy of ergodicity and mixing conditions.
We will not study this issue further in this paper, and instead now
turn to another aspect of joinings.

\section{Mixing and multi-time correlations functions}

A variety of mixing (or clustering) conditions for quantum systems
have appeared in the physics literature (see for example \cite{NT}),
and this is related to so-called multi-time correlations which have
been studied in \cite{BF, ABDF, ABDF2} using free products of
operator algebras. In this section we first study higher order
mixing of strongly mixing group systems (as defined in the previous
section) as a motivating example for multi-time correlation
functions, and then we show more generally how multi-time
correlation functions fit into a joining framework, although we work
in a slightly simplified setting compared to above mentioned physics
literature in order to make the connection with joinings very clear.
Throughout this section all systems have $G=\mathbb{Z}$. We will
often use the notation $[n]:=\{1,...,n\}$ for
$n\in\mathbb{N}=\{1,2,3,...\}$.

Recall that a system $\mathbf{B}=(B,\nu,\beta)$ is called
\emph{strongly mixing} if
\[
\lim_{n\rightarrow\infty}\nu\left(a\beta^{n}(b)\right)=\nu(a)\nu(b)
\]
for all $a,b\in B$. It turns out that a group system is strongly
mixing if and only if it is ergodic (see for example \cite[Theorem
3.4]{D2}). We now show that a strongly mixing group system is
$k$-mixing for all $k\in\mathbb{N}$, i.e. ``mixing of all orders'',
and we formulate it in terms of joinings (essentially the same
result is quoted in \cite[Section 3.1]{BF} in a slightly
different context and not in the language of joinings). For $\bar{n}%
=(n_{1},...,n_{k})\in\mathbb{Z}^{k}$ we use the notation $\bar{n}%
\rightarrow\infty$ to mean
$n_{1}\rightarrow\infty,n_{2}-n_{1}\rightarrow
\infty,...,n_{k}-n_{k-1}\rightarrow\infty$.

\begin{theorem}
Let $\mathbf{B}$ be a strongly mixing group system and consider any
$k\in\mathbb{N}$. For $\bar{n}\in\mathbb{Z}^{k}$, let
$\Delta_{\bar{n}}$ be the joining given by Construction 2.3 in terms
of $I=[k]$ and $\mathbf{A}_{\iota}=\mathbf{B}$. Then
\[
\lim_{\bar{n}\rightarrow\infty}\Delta_{\bar{n}}(a)=\left(\ast_{\iota\in
I}\nu\right)(a)
\]
for all $a\in\ast_{\iota\in I}B$.
\end{theorem}

\begin{proof}
Let the dynamics $\beta$ be given by the automorphism $T$ of the
free group $\Gamma$ obtained from a bijection of the alphabet $S$ of
symbols of the group, as explained in Section 3. It is convenient to
work explicitly in terms of the index $\iota$, e.g. the left regular
representation $\lambda_{\iota}=\lambda$ of $\Gamma$ will be indexed
by $\iota$ when we view $\lambda(g)$ as an element of $A_{\iota}$.
Consider any $g_{1},...,g_{m}\in\Gamma \backslash\{1\}$ and
$\iota_{1},...,\iota_{m}\in I$ with $\iota_{j}\neq \iota_{j+1}$.
Since $\mathbf{B}$ is strongly mixing, all the orbits of $T$ must be
infinite, except on the identity of $\Gamma$. Hence for $\bar{n}$
``large enough'' (in the sense of $\bar{n}\rightarrow\infty$)\ the
group elements $T^{n_{\iota_{p}}}g_{p}$ and $T^{n_{\iota_{q}}}g_{q}$
will have no symbols in common for any $\iota_{p}\neq\iota_{q}$ and
therefore
\[
\Delta_{\bar{n}}\left(\lambda_{\iota_{1}}(g_{1})...\lambda_{\iota_{m}}(g_{m})\right)
=\nu\left(\lambda\left(T^{n_{\iota_{1}}}g_{1}
...T^{n_{\iota_{m}}}g_{m}\right)  \right)  =0
\]
but $\left(\ast_{\iota\in
I}\mu_\iota\right)\left(\lambda_{\iota_{1}}(g_{1})...\lambda_{\iota_{m}}(g_{m})\right)=0$,
since $\mu_{\iota_j}\left( \lambda_{\iota_{j}}(g_{j})\right) =0$. We
conclude that for any $a$ in a dense subset of $\ast_{\iota\in
I}A_{\iota}$ we have $\Delta_{\bar{n}}(a)=\left( \ast_{\iota\in
I}\mu_\iota\right)(a)$ for $\bar{n}$ large enough, and the result
follows.
\end{proof}

This theorem implies for example that
\[
\lim_{\bar{n}\rightarrow\infty}\nu(\beta^{n_1}(a_1)...\beta^{n_k}(a_k))=\nu(a_1)...\nu(a_k)
\]
for all $a_1,...,a_k \in B$, which is why we view it as expressing
$k$-mixing.

More generally consider an arbitrary system
$\mathbf{B}=(B,\nu,\beta)$ and a fixed $k\in\mathbb{N}$. For any
$a_{1},...,a_{m}\in B$ and $\iota_{1},...,\iota _{m}\in\lbrack k]$
with $\iota_{j}\neq\iota_{j+1}$ for all $j$ we call
\[
\mathbb{Z}^{k}\ni\bar{n}\mapsto
\nu\left(\beta^{n_{\iota_{1}}}(a_{1})...\beta^{n_{\iota_{m}}}(a_{m})\right)
\]
a \emph{multi-time correlation function} of $\mathbf{B}$. All of
these multi-time correlation functions are subsumed in the single
function $\mathbb{Z}^{k}\ni\bar{n}\mapsto\Delta_{\bar{n}}$ where
$\Delta_{\bar{n}}$ is again the joining obtained in Construction 2.3
in terms of $I=[k]$ and $\mathbf{A}_{\iota }=\mathbf{B}$. So our
first conclusion is that multi-time correlation functions are in
fact given by joinings. Furthermore, in terms of this notation we
have the following simple theorem regarding an average of the
$\Delta_{\bar{n}}$ 's:

\begin{theorem}
Let $\left(\Phi_{N}\right)_{N\in\mathbb{N}}$ be a F\o lner sequence
in the group $\mathbb{Z}^{k}$. Then
\[
\bar{\Delta}_{N}:=\frac{1}{\left|\Phi_{N}\right|}\sum_{\bar{n}\in\Phi_{N}}\Delta_{\bar{n}}
\]
is a joining of $\left(\mathbf{A}_{\iota}\right)_{\iota\in I}$ for
every $N\in\mathbb{N}$. If
$\omega(a):=\lim_{N\rightarrow\infty}\bar{\Delta}_{N}(a)$ exists for
every $a\in\ast_{\iota\in I}B$ then $\omega$ is a joining of
$\left(\mathbf{A}_{\iota}^{\prime}\right)_{\iota\in I}$ where
$\mathbf{A}_{\iota}^{\prime}:=\left(B,\nu,\beta^{m_{\iota}}\right)$
for any $m_{\iota}\in\mathbb{Z}$.
\end{theorem}

\begin{proof}
The first part is clear. So assume $\omega(a)$ exists. Then it is
clear that $\omega$ is a state on $\ast_{\iota\in I}B$, and since
$\bar{\Delta}_{N}$ is a joining, we see that
$\omega\circ\psi_{\iota}=\nu$. Let $\tau$ denote dynamics on
$\ast_{\iota\in I}A$ obtained from the $\beta^{m_{\iota}}$ 's.
Setting $\bar{m}:=\left(  m_{1},...,m_{k}\right)  $ and using the
universal property of the free product one easily finds that
\[
\bar{\Delta}_{N}\circ\tau=\frac{1}{\left|  \Phi_{n}\right|
}\sum_{\bar{n}\in\Phi_{N}+\bar{m}}\Delta_{\bar{n}}
\]
and since $\left(  \Phi_{N}\right)_{N\in\mathbb{N}}$ is a F\o lner
sequence, which means that
\[
\left|  \Phi_{N}\bigtriangleup\left(\Phi_{N}+\bar {m}\right) \right|  /\left|  \Phi_{N}\right|
\rightarrow0
\]
as $N\rightarrow \infty$, it follows that $\omega\circ\tau=\omega$
so $\omega$ is indeed a joining of $\left(
\mathbf{A}_{\iota}^{\prime}\right) _{\iota\in I}$.
\end{proof}

The joining $\omega$ in Theorem 4.2 is a simplified version of the
``asymptotic state'' considered in \cite{ABDF}, and the fact that it
is a joining as described, corresponds to part of \cite[Proposition
3.1]{ABDF}. In that paper they however use a countably infinite free
product instead of $\ast_{\iota \in\lbrack k]}A_\iota$ as we did, to
allow for variable $k$, and they use a more abstract averaging
procedure. The essential point remains the same though.

Note that Theorem 4.1 provides an illustration of Theorem 4.2: It is
easy to see that $\left(  [N]+sN\right)  _{N\in\mathbb{N}}$ is a F\o
lner sequence in $\mathbb{Z}$ for any $s\in\mathbb{Z}$, and since
the cartesian product of the terms of F\o lner sequences leads to a
F\o lner sequence in the cartesian product of the involved groups,
we see that
\[
\Phi_{N}:=\left(  [N]+2N\right)  \times\left(  \lbrack N]+4N\right)
\times...\times\left(  \lbrack N]+2kN\right)
\]
provides us with a F\o lner sequence in $\mathbb{Z}^{k}$ for which
$\omega(a):=\lim_{N\rightarrow\infty}\bar{\Delta}_{n}(a)=\left(\ast_{\iota\in
I}\nu\right)(a)$ is simple to verify for the situation in Theorem
4.1. In this case of course $\omega=\ast_{\iota\in I}\nu$ is
trivially a joining of $\left( \mathbf{A}_{\iota}^{\prime}\right)
_{\iota\in I}$.

\section*{Acknowledgment} I thank Anton Str\"{o}h for encouraging me
to study free products of operator algebras, and the National
Research Foundation for financial support.

\end{document}